\documentclass{amsart}[12pt]
\usepackage{amssymb,amsmath}
\usepackage{enumerate}
\usepackage[english]{babel}
\usepackage{color}

\newcommand {\ep} {\varepsilon}

\newcommand {\gm} {\gamma}
\newcommand {\ii} {\infty}
\newcommand {\dt} {\delta}
\newcommand {\al} {\alpha}
\newcommand {\bt} {\beta}
\newcommand {\lb} {\lambda}
\newcommand {\Lb} {\Lambda}

\newcommand {\su} {\subset}
\newcommand {\wt} {\widetilde}
\newcommand {\wh} {\widehat}
\newcommand {\pr} {\prime}
\newcommand {\mc} {\mathcal}
\newcommand {\pp} {\perp}
\newcommand {\mbb} {\mathbb}

\newtheorem{teo}{Theorem}[section]
\newtheorem{pro}{Proposition}[section]
\newtheorem{cor}{Corollary}[section]

\newtheorem{lm}{Lemma}[section]

\theoremstyle{definition}
\newtheorem{rem}{Remark}[section]
\newtheorem{df}{Definition}[section]

\title{Notes on noncommutative ergodic theorems}
\begin{document}
\date{July 30, 2023}

\begin{abstract}
Given a semifinite von Neumann algebra $\mc M$ equipped with a faithful normal semifinite trace $\tau$, we prove that the spaces $L^0(\mc M,\tau)$ and $\mc R_\tau$ are complete with respect to pointwise, almost uniform and bilaterally almost uniform, convergences in $L^0(\mc M,\tau)$. Then we show that the pointwise Cauchy property for a special class of nets of linear operators in the space $L^1(\mc M,\tau)$ can be extended to pointwise convergence of such nets in any fully symmetric space $E\su\mc R_\tau$, in particular, in any space $L^p(\mc M,\tau)$, $1\leq p<\ii$. Some applications of these results in the noncommutative ergodic theory are discussed.
\end{abstract}

\author{Semyon Litvinov}
\address{76 University Drive, Pennsylvania State University, Hazleton 18202}
\email{snl2@psu.edu}

\maketitle

\section{Introduction and Preliminaries}

Let $\mc M$ be a semifinite von Neumann algebra equipped with a faithful normal semifinite trace $\tau$, and let $\mc P(\mc M)$ stand for the complete lattice of projections in $\mc M$; see \cite{br, ta,sz}. If $\mathbf 1$ is the identity of $\mc M$ and $e\in \mc P(\mc M)$, we write $e^{\perp}=\mathbf 1-e$. In what follows, we assume that $\tau(\mathbf 1)=\ii$.
\vskip3pt
Denote by $L^0=L^0(\mc M,\tau)$ the algebra of $\tau$-measurable operators affiliated with $\mc M$. Let $L^p=L^p(\mc M,\tau)$ if $1\leq p<\ii$ and $L^\ii(\mc M,\tau)=\mc M$
be the noncommutative $L^p$-\,spaces associated with $(\mc M,\tau)$. For detailed accounts on the spaces $L^p(\mc M,\tau)$, $p\in \{0\}\cup [1,\ii)$, see \cite{se, ne, ye0, px}.

\vskip5pt
Let $(\Lb,\leq)$ be a directed set, and let $\| \cdot \|_\ii$ be the uniform norm in $\mc M$. The following definition presents two noncommutative counterparts of the classical pointwise (almost everywhere) convergence.

\begin{df}
A net $\{ x_\al\}_{\al\in\Lb}\su L^0$ is said to converge {\it almost uniformly} (a.u.) ({\it bilaterally almost uniformly} (b.a.u.)) to $\widehat x\in L^0$ if for any given $\ep>0$ there exists a projection $e\in \mc P(\mc M)$ such that $\tau(e^{\perp})\leq \ep$ and
\[
 \| (\wh x-x_\al)e\|_\ii\to 0\text{ \ (\,respectively,\ }\| e(\wh x-x_\al)e\|_\ii\to 0\,).
\]
\end{df}

\begin{rem} 
(1) A.u. convergence clearly implies b.a.u. convergence, while the converse in not true in general - see \cite[Example 3.1]{cl1}.

\noindent
(2) Pointwise convergence in the space $L^0$ is not topological in general, that is, it may not agree with any topology in $L^0$; see, for example, the note \cite{or}. 
\end{rem}

First ergodic theorem in the space $L^1$ equipped with b.a.u. convergence appeared in the groundbreaking paper \cite{ye}. Ergodic theorems in the spaces $L^p$, $1\leq p<\ii$, endowed with b.a.u. or a.u. for $p\ge 2$ convergence were comprehensively studied in another fundamental paper \cite{jx}. An alternative approach, based directly on \cite{ye}, to deriving ergodic theorems in the spaces $L^p$, $1<p<\ii$, was suggested in \cite{li}.

\vskip5pt
Given $x\in L^0$ and $t>0$, denote
\[
\mu_t(x)=\inf\big\{ \|xe\|_\ii:\, e\in\mc P(\mc M),\, \tau(e^\pp)\leq t\big\},
\]
the {\it $t$-th generalized singular number of $x$}\,-\,see \cite{fk}. The function $\mu_t(x)$, $t>0$, is called the {\it non-increasing rearrangement} of $x$. It is worth mentioning that, as it follows from the proof of \cite[Proposition 2.2]{fk},
\[
\mu_t(x)=\inf\big\{ \|xe\|_\ii:\, e\in\mc P(\mc A),\, \tau(e^\pp)\leq t\big\},
\]
where $\mc A$ is any subalgebra of $\mc M$ containing the spectral family $\big\{e_\lb(|x|)\big\}_{\lb\ge0}$ of the absolute value 
$|x|=(x^*x)^{1/2}$ of an operator $x\in L^0$. Thus, $\mu_t(x)$ can be alternatively defined as
\[
\mu_t(x)=\inf\big\{ \lb>0:\, \tau(e_\lb(|x|)^\pp)\leq t\big\}.
\]

\noindent
Now, let us denote
\[
\mc R_\tau=\big\{x\in L^1+\mc M:\ \mu_t(x)\to 0\text{\,\ as\ }t\to\ii\big\}.
\]
\vskip3pt
\noindent
In view of \cite[Lemma 2.5]{fk}, $\mc R_\tau$ is a $\ast$-algebra. Besides, it is clear that 
\[
\mc R_\tau=\big\{x\in L^1+\mc M:\ \tau(e_\lb(|x|)^\pp)<\ii\, \ \forall\ \lb>0\big\},
\]
so $\mc R_\tau$ is the noncommutative analogue of Fava's space $\mc R_0$ - see \cite[Section (2.2.15)]{es}.

\vskip5pt
Sequential completeness of the space $L^0$ with respect to a.u. and b.a.u. convergences was proved in \cite[Theorem 2.3]{cls}. In Section 2, we show that the space $L^0$ and - more importantly for this article - the space $\mc R_\tau$ are complete with respect to these convergences. 

Section 3 is devoted to the description how the a.u. (b.a.u.) convergence of a generic class of nets of linear maps in any fully symmetric space $E\su\mc R_\tau$ - in particular, in any space $L^p$, $1\leq p<\ii$ - can be deduced from the a.u. (respectively, b.a.u.) Cauchy property for such nets in the space $L^1$. 

In Section 4 we proceed to outline some possible applications of these results to the nocommutative ergodic theory.

\section{Pointwise completeness in $L^0(\mc M,\tau)$}
Here we aim to show that any almost uniform or bilaterally almost uniform Cauchy net converges in $L^0$ and to derive the same for the space $\mc R_\tau$. We will employ the notion of measure topology in $L^0$ - in which $L^0$ is a complete metric space - and use the fact that a Cauchy net in a complete metric space is convergent.
\vskip3pt

The {\it measure topology}\, $t_\tau$ in $L^0$ is given by the system
\[
\mc N(\ep,\dt)=\big\{ x\in L^0: \ \exists\ e\in \mc P(\mc M) \text{\,\ such that\,\ } \tau(e^{\perp})\leq\ep\text{\,\ and\ } \| xe\|_\ii\leq \dt  \big\}\, ;
\]
$\ep>0$, $\dt>0$, of (closed) neighborhoods of zero. The following is an immediate corollary of \cite[Theorem 1]{ne}; note that it is understood that the space $L^0$ is the completion of $\mc M$ with respect to $t_\tau$.
\begin{teo}\label{t1}
$(L^0,t_\tau)$ is a complete metrizable topological $\ast$-algebra.
\end{teo}

\begin{df}
A net $\{ x_\al\}_{\al\in\Lb}\su L^0$ is said to converge to $\widehat x\in L^0$ {\it in measure (bilaterally in measure)} 
if for every $\ep>0$ and $\dt>0$ there is $\gm=\gm(\ep,\dt)\in I$ such that, given $\al\in \Lb$ with $\al\ge\gm$, there exists a projection $e_\al\in\mc P(\mc M)$ satisfying conditions 
$\tau(e^{\perp}_\al)\leq \ep$ and
\[
\| (\widehat x-x_\al)e_\al\|_\ii\leq\dt\,\ (\text{respectively}, \,\| e_\al(\widehat x-x_\al)e_\al\|_\ii\leq\dt).
\]
\end{df} 

Proof of the following lemma - which idea is attributed to E. Nelson (see \cite[Appendix, Theorem B]{ra}); note also the paragraph before \cite[Theorem 1]{ye} - can be found in \cite[Lemma 2.1, Theorem 2.2]{cls}.
\begin{lm}\label{l1}
Let $x\in L^0$ and $e\in\mc P(\mc M)$ be such that $exe\in\mc M$. Then there exists $f\in\mc P(\mc M)$ such that
\[
\tau(f^\pp)\leq2\tau(e^\pp)\text{ \ and \ } \|xf\|_\ii\leq\|exe\|_\ii.
\]
\end{lm}

The following statement is a version of \cite[Theorem 2.2]{cls} for nets. We provide a proof for the sake of the reader's convenience.
\begin{pro}\label{p2}
A net $\{ x_\al\}_{\al\in\Lb}\su L^0$ is Cauchy in measure if and only if it is Cauchy bilaterally in measure.
\end{pro}
\begin{proof}
Clearly, if is sufficient to show the "if" implication. So, let $\{ x_\al\}_{\al\in\Lb}\su L^0$ be Cauchy bilaterally in measure and fix $\ep>0$, $\dt>0$. Then there is $\gm=\gm(\ep,\dt)$ such that for any $\al,\bt\in I$ with $\al,\bt\ge\gm$ there exists $e_{\al,\bt}\in\mc P(\mc M)$ for which
\[
\tau(e_{\al,\bt}^\pp)\leq\frac\ep2\text{ \ and \ } \|e_{\al,\bt}(x_\al-x_\bt)e_{\al,\bt}\|_\ii\leq\dt.
\]
Then, by Lemma \ref{l1}, for every $\al,\bt\ge\gm$ there exists $f_{\al,\bt}\in\mc P(\mc M)$ such that
\[
\tau(f_{\al,\bt}^\pp)\leq2\tau(e_{\al,\bt}^\pp)\leq\ep\text{ \ and \ } \|(x_\al-x_\bt)f_{\al,\bt}\|_\ii\leq\|e_{\al,\bt}(x_\al-x_\bt)e_{\al,\bt}\|_\ii\leq\dt,
\]
implying that the net $\{ x_\al\}_{\al\in\Lb}\su L^0$ is Cauchy in measure.
\end{proof}

Now we are ready to prove main results of this section:
\begin{teo}\label{t2}
The space $L^0$ is complete with respect to a.u. and b.a.u. convergences.
\end{teo}
\begin{proof}
We will show that $L^0$ is b.a.u. complete. Proof in the a.u. case is similar. Let $\{x_\al\}_{\al\in\Lb}\su L^0$ be b.a.u. Cauchy.
Then $\{x_\al\}_{\al\in\Lb}$ is clearly Cauchy bilaterally in measure, which, by Proposition \ref{p2}, implies that it is Cauchy in measure. 

Let $d$ be a metric in $L^0$ compatible with the measure topology $t_\tau$. Then for every $\dt>0$ there is $\gm\in I$ such that
$d(x_\al,x_\bt)\leq\dt$ for all $\al,\bt\ge\gm$. 

So, pick $\gm_1\in I$ such that $d(x_\al,x_\bt)\leq1$ whenever $\al,\bt\ge\gm_1$. Further, given $\gm_n\in I$ such that $\displaystyle d(x_\al,x_\bt)\leq\frac1n$ for all $\al,\bt\ge\gm_n$, choose $\gm_{n+1}\in I$ such that
\[
\gm_{n+1}\ge\gm_n\text{\, \ and\, \ }d(x_\al,x_\bt)\leq\frac1{n+1}\ \ \forall\ \al,\bt\ge\gm_{n+1}.
\] 
It follows that the sequence $\{x_{\gm_n}\}$ is Cauchy in measure, hence, in view of Theorem \ref{t1}, there exists $\wh x\in L^0$ such that $d(x_{\gm_n},\wh x)\to 0$ as $n\to\ii$. 

Given $\dt>0$, there is $n_1$ such that $d(x_{\gm_n},\wh x)\leq\dt/2$ whenever $n\ge n_1$. In addition, there is $n_2$ such that $\displaystyle d(x_\al,x_\bt)\leq\frac\dt2$ whenever $\al,\bt\ge\gm_{n_2}$. If $n_0=\max\{n_1,n_2\}$, then  
\[
d(x_{\gm_{n_0}},\wh x)\leq\frac\dt2\text{\ \ \ and, as\ } \gm_{n_0}\ge\gm_{n_1}, \ d(x_\al,x_{\gm_{n_0}})\leq\frac\dt2 \ \ \forall\ \al\ge\gm_{n_0}.
\]
Therefore\, $d(x_\al,\wh x)\leq\dt$ whenever $\al\ge\gm_{n_0}$, implying that $x_\al\to \wh x$ in measure.

Finally, fix $\ep>0$. Since $\{x_\al\}_{\al\in\Lb}$ is b.a.u. Cauchy, there exists $e\in\mc P(\mc M)$ with $\tau(e^\pp)\leq\ep$ such that the net $\{ex_\al e\}_{\al\in\Lb}$ is $\|\cdot\|_\ii$\,-\,Cauchy. Then, as $(\mc M, \|\cdot\|_\ii)$ is a complete metric space, we show as above that there is $x_e\in\mc M$ such that $\|ex_\al e-x_e\|_\ii\to 0$, implying that $ex_\al e\to x_e$ in measure. Also, as $x_\al\to\wh x$ in measure, it follows from Theorem \ref{t1} that $ex_\al e\to e\wh xe$ in measure, hence $x_e=e\wh xe$, and we obtain
\[
\|e(\wh x-x_\al)e\|_\ii=\|x_e-ex_\al e\|_\ii\to 0.
\]
Therefore, it follows that $L^0$ is b.a.u. complete.
\end{proof}

\begin{teo}\label{t3}
The space $\mc R_\tau$ is complete with respect to a.u. and b.a.u. convergences.
\end{teo}
\begin{proof}
In view of Theorem \ref{t2} and, since $\mc R_\tau\su L^0$ and a.u. convergence implies b.a.u. convergence, it is sufficient to show that $\mc R_\tau$ is b.a.u. closed.

Let a net $\{x_\al\}_{\al\in I}\su\mc R_\tau$ be such that $x_\al\to\wh x\in L^0$ b.a.u. Since $\mu_t(x)$ is a non-increasing function of $t$, in order to show that $\wh x\in\mc R_\tau$, that is, that $\mu_t(x)\to 0$ as $t\to\ii$, it is sufficient to verify that for every $\dt>0$ there is $t>0$ such that $\mu_t(x)<\dt$. Since $\|xe\|_\ii<\dt$ for some $e\in\mc P(\mc M)$ with $\tau(e^\pp)\leq t$ implies that $\mu_t(x)<\dt$, this in turn will follow if we show that there exists $f\in\mc P(\mc M)$ such that 
\[
\tau(f^\pp)<\ii\text{ \ and \ }\|\wh xf\|_\ii<\dt.
\]

Since $x_\al\to\wh x$\, b.a.u., given $n\in\mathbb N$, there exits $e_n\in\mc P(\mc M)$ such that 
\[
\tau(e_n^\pp)\leq2^{-n}\text{ \ and \ } \|e_n(\wh x-x_\al)e_n\|_\ii\xrightarrow[\al]\ 0.
\]
Therefore, for each $n$ there is $x_n\in\{x_\al\}_{\al\in I}$ for which $\|e_n(\wh x-x_n)e_n\|_\ii\leq\displaystyle\frac1n$. 

Fix $\ep>0$, and let $N$ be such that $\sum\limits_{n=N}^\ii2^{-n}\leq\ep$. Then, with $e=\bigwedge\limits_{n\ge N}e_n$, we have
\[
\tau(e^\pp)\leq\ep\text{ \  and \ }\|e(\wh x-x_n)e\|_\ii\leq\frac1n\ \ \forall \ n\ge N, 
\]
implying that $\mc R_\tau\ni x_n\to\wh x$\, b.a.u.

Now, let $\dt>0$. Since $x_1\in\mc R_\tau$, there is $e_1\in\mc P(\mc M)$ with $\tau(e_1^\pp)<\ii$ such that $\|x_1e_1\|_\ii<\dt$. As $x_2\in L^0$, there exists 
$\mc P(\mc M)\ni h_1\leq e_1$ such that $\tau(h_1^\pp)<\ii$ and $x_2h_1\in\mc M$. Next, $x_2\in\mc R_\tau$ entails that there is $\mc P(\mc M)\ni e_2\leq h_1$ such that
$\tau(e_2^\pp)<\ii$ and $\|x_2e_2\|_\ii<\dt$. Repeating this process, one can construct a sequence $\{e_n\}\su\mc P(\mc M)$ satisfying conditions 
\[
e_n\ge e_{n+1},\, \ \tau(e_n^\pp)<\ii,\text{\ and \ }\|x_ne_n\|_\ii<\dt,\text{\,hence\ }\|e_nx_ne_n\|_\ii<\dt,\, \ \forall\ n.
\]

Further, $x_n\to\wh x$ b.a.u. implies that there exists $\mc P(\mc M)\ni e\leq e_1$ such that $\tau(e-e_1)<\ii$, hence $\tau(e^\pp)<\ii$, and $\|e(\wh x-x_n)e\|_\ii\to 0$, implying that
\[
\|(e_n\land e)\wh x(e_n\land e)\|_\ii-\|(e_n\land e)x_n(e_n\land e)\|_\ii\to 0.
\]
Therefore, since $\|(e_n\land e)x_n(e_n\land e)\|_\ii<\dt$ for every $n$, it follows that there exists $n_0$ such that $\|(e_{n_0}\land e)\wh x(e_{n_0}\land e)\|_\ii<\dt$. Then, by Lemma \ref{l1}, it is possible to find $f\in\mc P(\mc M)$ such that 
\[
\tau(f^\pp)\leq2\tau((e_{n_0}\land e)^\pp)\leq2\tau(e_{n_0}^\pp)+2\tau(e^\pp)<\ii\text{ \ and \ }\|\wh xf\|_\ii<\dt,
\]
which completes the argument.
\end{proof}

\section{Pointwise convergence in fully symmetric spaces $E\su\mc R_\tau$}
A non-zero Banach space $(E,\|\cdot\|_E)\su L^0$ is called {\it fully symmetric} if conditions
\[
x\in E, \ y\in L^0,\text{\ and\ }\int_0^s\mu_t(y)dt\leq\int_0^s\mu_t(x)dt\ \ \forall\ s>0\text{\ (denoted\ } y\prec\prec x)
\]
imply that $y\in E$ and $\|y\|_E\leq\|x\|_E$.
\vskip5pt
It is well-known that the space $L^p=(L^p(\mc M,\tau), \|\cdot\|_p)$ is fully symmetric and $L^p\su\mc R_\tau$ for each $1\leq p<\ii$. Besides, if $E\su L^0$ is a fully symmetric space, then $E\su\mc R_\tau$ if and only if $\mathbf 1\notin E$ \cite[Proposition 2.2]{cl1}, implying that the noncommutative Orlicz, Lorentz and Marcinkiewicz spaces $E$ on $(\mc M,\tau)$ with $\mathbf 1\notin E$ also lie in $\mc R_\tau$; see, for example, \cite[Section 7]{cl2}.

\begin{rem}
In contrast with Theorem \ref{t3}, even in the commutative case, a fully symmetric space  may not be complete with respect to almost uniform convergence.  
Indeed, denote by $\nu$ Lebesgue measure on the interval $(0,\ii)$ and define
\[
\mc R_\nu=\big\{\, f\in L^1(\nu)+L^\ii(\nu): \  \nu \{|f|>\lb\}<\ii\ \ \forall \ \lb>0\, \big\},
\]

\noindent
the commutative precursor of $\mc R_\tau$, known as Fava's space $\mc R_0$; see \cite{cl4} for details.

\vskip5pt
Given\,\ $t\in (0,\ii)$, set
\[
f_n(t)=\sum_{k=1}^n2^k\cdot\chi_{(2^{-k},\, 2^{-k+1}]}(t),\ \, n=1,2,\dots
\]
It is easy to see that $\{f_n\}\su L^1(\nu)$ and that $\{f_n\}$ converges almost uniformly to a function $f\in\mc R_\nu$ such that $f\notin L^1(\nu)$. Therefore, the fully symmetric space\, $L^1(\nu)\su\mc R_\nu$ is not complete with respect to almost uniform convergence. However, for a special class of pointwise Cauchy nets in a fully symmetric space $E$ the limits belong to $E$ - see Theorem \ref{t4} below.
\end{rem}

A linear operator $A: L^1+\mc M\to L^1+\mc M$ is called a {\it Dunford-Schwartz operator} (denoting $A\in DS$) if $A(L^1)\su L^1$, $A(\mc M)\su\mc M$ and
\[
\|A(x)\|_1\leq\|x\|_1\, \ \forall\ x\in L^1\text{ \ and \ }\|A(x)\|_\ii\leq\|x\|_\ii\, \ \forall\ x\in \mc M.
\]

The following fact is well-known in the commutative case; see, for example, \cite[Chapter II, Section 4]{kps}.
\begin{pro}\label{p3}
If $A\in DS$, then $A(x)\prec\prec x$ for all $x\in L^1+\mc M$.
\end{pro}
\begin{proof}
As it is proven right after \cite[Theorem 4.4]{fk} - see also \cite[Proposition 2.5]{ddp} - given $x\in L^1+\mc M$ and $s>0$, we have
\[
\int_0^s\mu_t(x)ds=\inf\big\{\|y\|_1+s\|z\|_\ii:\ x=y+z,\ x\in L^1,\ y\in \mc M\big\}.
\]
Let $x\in L^1+\mc M$ and $s>0$. As $x=y+z$, with $y\in L^1$ and $z\in\mc M$, implies that $A(x)=A(y)+A(z)$, where
$A(y)\in L^1$ and $A(z)\in\mc M$, we can write
\[
\begin{split}
\int_0^s\mu_t(A(x))ds&=\inf\big\{\|y\|_1+s\|z\|_\ii:\ A(x)=y+z,\ x\in L^1,\ y\in \mc M\big\}\\
&\leq\inf\big\{\|A(y)\|_1+s\|A(z)\|_\ii:\ A(x)=A(y)+A(z),\ x\in L^1,\ y\in \mc M\big\}\\
&\leq\inf\big\{\|y\|_1+s\|z\|_\ii:\ A(x)=A(y)+A(z),\ x\in L^1,\ y\in \mc M\big\}\\
&\leq\inf\big\{\|y\|_1+s\|z\|_\ii:\ x=y+z,\ x\in L^1,\ y\in \mc M\big\}\\
&=\int_0^s\mu_t(x)ds,
\end{split}
\]
so the result follows.
\end{proof}

The following statement provides a useful characterization of the $*$-algebra $\mc R_\tau$.
\begin{pro}\label{p1}
Let $x^*=x\in L^0$. Then $x\in\mc R_\tau$ if and only if for any $\dt>0$ there exist $y_\dt\in L^1$ and $z_\dt\in\mc M$ such that
\begin{equation}\label{e1}
\|z_\dt\|_\ii\leq\dt\text{ \ and \ }x=y_\dt+z_\dt.
\end{equation}
\end{pro}
\begin{proof}
$"\Rightarrow" :\ $ Fix $\dt>0$. Since $|x|\in\mc R_\tau$, we have $0\leq x_+=\displaystyle\frac{|x|+x}2\in\mc R_\tau$ and  $0\leq x_-=\displaystyle\frac{|x|-x}2\in\mc R_\tau$. If $x_+=\int_0^\ii\lb de_\lb$ is the spectral decomposition of $x_+$, define $y_\dt^+=\int_{\dt/2}^\ii\lb de_\lb$ and $z_\dt^+=\int_0^{\dt/2}\lb de_\lb$. Then we have 
\[
x_+=y_\dt^++z_\dt^+\text{ \ and \ }\|z_\dt^+\|_\ii\leq\frac\dt2.
\]
Next, $x_+\in\mc R_\tau$ implies that $\tau(e_{\dt/2}^\pp)<\ii$ and $x_+=y^++z^+$ for some $y^+\in L^1$ and $z^+\in\mc M$, hence
\[
y_\dt^+=y_\dt^+e_{\dt/2}^\pp=y^+e_{\dt/2}^\pp+(z^+-z_\dt^+)e_{\dt/2}^\pp\in L^1.
\]
Similarly, $x_-=y_\dt^-+z_\dt^-$, where $y_\dt^-\in L^1$ and $\|z_\dt^-\|_\ii\leq\dt/2$, so it follows that
\[
x=(y_\dt^+-y_\dt^-)+(z_\dt^+-z_\dt^-),\text{ \ where \ }y_\dt^+-y_\dt^-\in L^1\text{ \ and \ }\|z_\dt^+-z_\dt^-\|_\ii\leq\dt.
\]

$"\Leftarrow" :\ $ Let $x^*=x\in L^0$ be such that for any $\dt>0$ there exist $y_\dt\in L^1$ and $z_\dt\in\mc M$ for which conditions (\ref{e1}) hold. Fix $\dt>0$. There exist $y_\dt\in L^1$ and $z_\dt\in\mc M$ such that $\|z_\dt\|_\ii\leq\dt/2$ and 
$x=y_\dt+z_\dt$. Then, by \cite[Lemma 2.5]{fk}, we have 
\[
\mu_t(x)\leq\mu_{t/2}(y_\dt)+\mu_{t/2}(z_\dt)\leq\mu_{t/2}(y_\dt)+\frac\dt2
\]
whenever $t>\dt$. Also, since $y_\dt\in L^1$, it follows that $\mu_t(y_\dt)\to 0$ as $t\to\ii$, so there is $t_0>0$ such that $t\ge t_0$ entails $\mu_t(y_\dt)\leq\dt/2$. Thus, we conclude that $\mu_t(x)\leq\dt$ for all $t>\max\{\dt,t_0\}$, implying that 
$\mu_t(x)\to 0$ as $t\to\ii$, so $x\in\mc R_\tau$.
\end{proof}
\vskip5pt
We will call a map $A: L^0\to L^0$ {\it selfadjoint} if $A(x)^*=A(x)$ whenever $x^*=x$. Here is an immediate consequence of Proposition \ref{p1}:

\begin{cor}\label{c1}
If a linear selfadjoint map $A: L^1+\mc M\to L^1+\mc M$ is such that $A(L^1)\su L^1$, 
$A(\mc M)\su\mc M$, and $\|A(x)\|_\ii\leq c\|x\|_\ii$ for some $c>0$ and any $x\in\mc M$, then $A(\mc R_\tau)\su\mc R_\tau$.
\end{cor}
\begin{proof}
Let $x\in\mc R_\tau$ and fix $\dt>0$. By Proposition \ref{p1}, there exist $y_\dt\in L^1$ and $z_\dt\in\mc M$ such that 
$\|z_\dt\|_\ii\leq\dt/c$ and $x=y_\dt+z_\dt$. Then $A(x)=A(y_\dt)+A(z_\dt)$, where $A(y_\dt)\in L^1$ and $\|A(z_\dt)\|_\ii\leq\dt$, and Proposition \ref{p1} implies that $A(x)\in\mc R_\tau$.
\end{proof}

\begin{teo}\label{t4}
Let a net of linear selfadjoint maps $A_\al: L^1+\mc M\to L^1+\mc M$, $\al\in\Lb$, be such that
\begin{enumerate}
\item $A_\al(L^1)\su L^1$ for every $\al\in\Lb$, and the net $\{A_\al(x)\}_{\al\in\Lb}$ is a.u. (b.a.u.) Cauchy for each $x\in L^1$;
\item $A_\al(\mc M)\su\mc M$ for every $\al\in\Lb$, and there exists $c_\ii>0$ such that 
\[
\sup_{\al\in\Lb}\|A_\al(x)\|_\ii\leq c_\ii\|x\|_\ii\ \ \forall \ x\in\mc M.
\] 
\end{enumerate}
Then $A_\al(x)\to\wh x\in\mc R_\tau$ a.u. (respectively, b.a.u.) for every $x\in\mc R_\tau$. If, in addition, $\|A_\al(x)\|_1\leq c_1\|x\|_1$ for some $c_1>0$ and all $x\in L^1$, then, given a fully symmetric space $E\su\mc R_\tau$, we have $A_\al(E)\su E$ for every $\al\in \Lb$ and $A_\al(x)\to\widetilde x\in E$ a.u. (respectively, b.a.u.) for every $x\in E$.
\end{teo}
\begin{proof}
We will prove the statement for a.u. convergence. Proof for b.a.u. convergence is similar. Assume, without loss of generality, that $x^*=x$ and fix $\ep>0$. By Proposition \ref{p1}, given $m\in\mathbb N$, there exist $y_m\in L^1$ and $z_m\in\mc M$ such that $\displaystyle\|z_m\|_\ii\leq\frac1m$ and $x=y_m+z_m$. Since the net $\{A_\al(y_m)\}_{\al\in\Lb}$ is a.u. Cauchy, there exist $e_m\in\mc P(\mc M)$ and $\gm_m\in I$ such that 
\[
\tau(e_m^\pp)\leq\frac\ep{2^m}\text{ \ \ and \ \ }\|(A_\al(y_m)-A_\bt(y_m))e_m\|_\ii\leq\frac1m\, \ \ \forall \ \al,\bt\ge\gm_m.
\]
Thus, letting $e=\bigwedge_me_m$, we obtain
\[
\tau(e^\pp)\leq\ep\text{ \ \ and \ \ }\|(A_\al(y_m)-A_\bt(y_m))e\|_\ii\leq\frac1m\,\ \ \forall \ \al,\bt\ge\gm_m.
\]
Therefore, given $\al,\bt\ge\gm_m$, it follows that
\[
\begin{split}
\|(A_\al(x)-A_\bt(x))e\|_\ii&=\|(A_\al(y_m)-A_\bt(y_m))e\|_\ii\\
&+\|(A_\al(z_m)-A_\bt(z_m))e\|_\ii\leq\frac{1+2c_\ii}m,
\end{split}
\]
so the net $\{A_\al(x)\}_{\al\in\Lb}$ is a.u. Cauchy. Since, by Corollary \ref{c1}, $\{A_\al(x)\}_{\al\in I}\su\mc R_\tau$, and, by Theorem \ref{t3}, the space $\mc R_\tau$ is a.u. complete, we conclude that there exists $\wh x\in\mc R_\tau$ such that 
$A_\al(x)\to\wh x\in\mc R_\tau$ a.u.

Let now $E\su\mc R_\tau$ be a fully symmetric space, and assume, in addition, that $\|A_\al(x)\|_1\leq c_1\|x\|_1$ for some $c_1>0$ and all $x\in L^1$. Put $c=\max\{c_\ii,c_1\}$ and denote $\wt A_\al=c^{-1}A_\al(x)$, $\al\in\Lb$. As $\{\wt A_\al\}_{\al\in I}\su DS$, we conclude, in view of Proposition \ref{p3}, that $\wt A_\al(x)\prec\prec x$ for all $x\in L^1+\mc M$, hence $\wt A_\al(E)\su E$, and so $A_\al(E)\su E$ for all $\al\in\Lb$. 

Next, let $x\in E$. We have $\wt A_\al(x)\to\wh x\in\mc R_\tau$ a.u. To simplify the argument, choose, as in the proof of Theorem \ref{t3}, a sequence $\{\wt A_n(x)\}\su\{\wt A_\al(x)\}$ with $\wt A_n(x)\to\wh x$ a.u. Then it is clear that $\wt A_n(x)\to\wh x$ in measure, which, by \cite[Lemma 3.4]{fk}, implies that $\mu_t(\wt A_n(x))\to\mu_t(\wh x)$ almost everywhere on $((0,\ii),\nu)$. Further, as $\wt A_n(x)\prec\prec x$, we have
\[
\int_0^s\mu_t(\wt A_n(x))dt\leq\int_0^s\mu_t(x)dt\ \ \ \forall \ s>0,\, n\in\mbb N.
\]
Therefore, as $\mu_t(\wt A_n(x))\to\mu_t(\wh x)$ almost everywhere on $((0,s),\nu)$, it follows from Fatou's lemma that
\[
\int_0^s\mu_t(\wh x)dt\leq\int_0^s\mu_t(x)dt\ \ \ \forall \ s>0,
\]
that is\, $\wh x\prec\prec x$. Since $E$ is a fully symmetric space, we conclude that $\wh x\in E$, hence $A_\al(x)\to c\,\wh x\in E$ a.u., completing the argument.
\end{proof}

\section{Applications}
For the sake of convenience, let us state the following immediate corollary of Theorem \ref{t4}.

\begin{pro}\label{c2}
Let $A_\al: L^1+\mc M\to L^1+\mc M$, $\al\in\Lb$, be a net of linear selfadjoint maps such that
\begin{enumerate}
\item there exist $c_1>0$ and $c_\ii>0$ satisfying conditions 
\[
\|A_\al(x)\|_1\leq c_1\|x\|_1\ \ \forall\ x\in L^1\text{\, \ and\, \ }\|A_\al(x)\|_\ii\leq c_\ii\|x\|_\ii\ \ \forall\ x\in\mc M;
\]
\item the net $\{A_\al(x)\}_{\al\in\Lb}$ is a.u. (b.a.u.) Cauchy for each $x\in L^1$.
\end{enumerate}
Then, given a fully symmetric space $E\su L^0$ with $\mathbf 1\notin E$ and $x\in E$, there exists $\wh x\in E$ such that 
$A_\al(x)\to\wh x$ a.u. (respectively, b.a.u.).
\end{pro}

We will now present some applications of Proposition \ref{c2}. 
\vskip5pt
Given $d\in\mathbb N$, denote 
\[
\mbb N^d_0=\{\, {\bf n}=(n_1,\dots,n_d):\, n_1,\dots,n_d\in\mathbb N\cup{0}\}.
\] 
Let $\big\{{\bf n}_\al=(n_1(\al),\dots,n_d(\al))\big\}_{\al\in\Lb}$ be a net in $\mbb N_0^d$. We write\, ${\bf n}_\al\to\ii$\, if\, $n_i(\al)\to\ii$ for each $1\leq i\leq d$. Denote $|{\bf n}_\al|^\pr=n_1(\al)\cdot\ldots\cdot n_d(\al)$, and let $|{\bf n}_\al|$ stand for $|{\bf n}_\al|^\pr$, with $n_i(\al)$ replaced by $1$ whenever $n_i(\al)=0$.

\vskip5pt
Assume that $\{T_1,\dots,T_d\}\su DS^+$ and let $x\in L^1+\mc M$.  If $\al\in\Lb$, set
\begin{equation}\label{e2}
A_\al(x)=\frac1{|{\bf n}_\al|}\sum_{k_1\leq n_1(\al),\dots,k_d\leq n_d(\al)}T_1^{k_1}\cdot\ldots\cdot T_d^{k_d}(x).
\end{equation}
Then $\{A_\al\}_{\al\in\Lb}$ is a net of linear selfadjoint (positive) maps in $L^1+\mc M$ that naturally satisfies condition (1) of Proposition \ref{c2}. To ensure that this net also satisfies condition (2) of Proposition \ref{c2}, we need an additional assumption that the net ${\bf n}_\al\to\ii$ remains in a sector of $\mbb N_0^d$ (see \cite[Ch.\,6, \S\,6.2, 6.3]{kr} for the sequential, commutative case):

\begin{df}
A net ${\bf n}_\al\to\ii$ is said to {\it remain in a sector} of $\mbb N_0^d$ if there exists $0<c_0<\ii$ such that 
\[
\displaystyle\frac{n_i(\al)}{n_j(\al)}\leq c_0\ \ \ \forall \ 1\leq i,j\leq d\text{\,\ and\,\ }\al\in\Lb,
\]
assuming that $n_j(\al)\neq 0$.
\end{df}
Now, straightforwardly adjusting to the net setting the argument in \cite{cl1} that yields \cite[Theorem 5.3]{cl1}, we derive from Yeadon's weak-type (1,1) maximal inequality \cite[Theorem 1]{ye} a weak-type $(1,1)$ maximal inequality for the net $\{A_\al\}_{\al\in\Lb}$ in (\ref{e2}) and use it to show that it is b.a.u. Cauchy for every $x\in L^1$. Thus, as a corollary of Proposition \ref{c2}, we arrive at the following (cf. \cite[Theorem 5.6]{cl1}).

\begin{teo}\label{t6}
Let $\{T_1,\dots,T_d\}\su DS^+$ be a commuting family, and let $E\su\mc R_\tau$ be a fully symmetric space, for example, any space $L^p$, $1\leq p<\ii$. Then, given $x\in E$, the net $\{A_\al(x)\}_{\al\in\Lb}$ given by (\ref{e2}) converges b.a.u. to some $\wh x\in E$, provided the net ${\bf n}_\al\to\ii$ remains in a sector of $\mbb N_0^d$.
\end{teo}

Theorem \ref{t6} for $d=1$ yields the following version of \cite[Theorem 2]{ye} (recall that the class of contractions $\alpha$ considered in \cite{ye} coincides with the class $DS^+$ modulo a unique extension - see \cite[Proposition 1.1]{cl}).

\begin{teo}
Let $T\in DS^+$, and let $\{n_\al\}_{\al\in\Lb}\su\mbb N$ be a net such that $n_\al\to\ii$. Then, given $x\in L^1$,  the net 
\[
A_\al(x)=\frac1{n_\al}\sum_{k=0}^{n_\al-1}T^k(x),  \ \al\in\Lb,
\]
converges b.a.u. to some $\wh x\in L^1$.
\end{teo}

Next, let $\mbb C_1= \{z\in \mbb C: |z|=1\}$. A function $p:\mbb R_+\to \mbb C$ is called a {\it trigonomertic polynomial} if $p(t)=\sum\limits_{j=1}^nw_j\lb_j^t$, where $n\in \mbb N$, $\{w_j\}_1^n\su\mbb C$, and $\{\lb_j\}_1^n\su \mbb C_1$. A Lebesgue measurable function $\bt: \mbb R_+\to\mbb C$ will be called a {\it bounded Besicovitch function} if  $\|\bt\|_\ii<\ii$, and for every $\ep>0$ there is a trigonometric polynomial $p_\ep$ such that
\[
\limsup_{t\to\ii}\frac 1t\int_0^t|\bt(s)-p_\ep(s)|ds<\ep.
\]

Let a semigroup $\{T_t\}_{t\ge 0}\su DS$ be {\it strongly continuous in} $L^1$, that is,
\[
\|T_t(x)-T_{t_0}(x)\|_1\to 0 \text{ \ \ whenever\ \ } t\to t_0 \text{ \ \ for all\ \ } x\in L^1.
\]
Then, as it is shown in \cite{cl3}, given a Lebesgue measurable function $\bt: \mbb R_+\to\mbb C$ with $\|\bt\|_\ii<\ii$, for any $x\in L^1$ and $t>0$ there exists the Bochner integral
\[
B_t(x)=\frac 1t\int_0^t\bt(s)T_s(x)ds \in L^1.
\]

Note that if\, $\|\beta\|_\ii\neq 0$, then for each $t>0$ the operator
\[
\|\beta\|_\ii^{-1}B_t:\, L^1\to L^1
\]
can be uniquely extended to a Dunford-Schwartz operator - see \cite{cl}. 

\vskip5pt
The following theorem was proved in \cite{cl3}. Here we present an alternative, straightforward, argument.
\begin{teo}\label{t5}
Let $\{T_t\}_{t\ge 0}\su DS^+$ be a strongly continuous semigroup in $L^1$, and let $\bt(t)$ be a bounded Besicovitch function. If $\mathbf 1\notin E\su L^0$ is a fully symmetric space, then, given $x\in E$, the flow $\{B_t(x)\}_{t>0}$ converges b.a.u. to some $\wh x\in E$ as $t\to\ii$. In particular, $B_t(x)\to\wh x\in L^p$ b.a.u. for every $x\in L^p$, $1\leq p<\ii$.
\end{teo}
\begin{proof}
It is clear that the net $\{B_t\}_{t>0}$ satisfies condition (1) of Proposition \ref{c2}. As the proof of \cite[Theorem 3.1]{cl3} shows that $\{B_t(x)\}_{t>0}$ is b.a.u. Cauchy for all $x\in L^1$, that is, condition (2) of Proposition \ref{c2} is also satisfied, the assertion follows.
\end{proof}

\begin{rem}
One can verify, using the approach in \cite{jx}, or in \cite{li, cl1}, that for $p\ge2$ the convergence in Theorems \ref{t6} and \ref{t5} is a.u.
\end{rem}

Finally, we notice that the results of this article can be applied to simplify and streamline some arguments in the noncommutative ergodic theory; for example, some presented in \cite{cl2, cl1}.

\vskip 5pt
\noindent
{\bf Acknowledgment}. The author would like to express his gratitude to Dr. Vladimir Chilin for his comments that helped to improve the article's presentation. The author is partially supported by the Pennsylvania State University traveling program.

\end{document}